\documentclass[11pt]{extarticle}

\usepackage{amsfonts, amssymb, amsmath, amsthm, textcomp, upgreek, stmaryrd, comment, dsfont, cmap, latexsym, ifthen}
\usepackage[utf8]{inputenc}
\usepackage[T2A]{fontenc}
\usepackage[english]{babel}
\usepackage[left=25mm, top=30mm, right=25mm, bottom=30mm, nohead, includefoot]{geometry}
\usepackage{xcolor}
\usepackage{indentfirst}

\def \geq {\geqslant}
\def \leq {\leqslant}
\newtheorem{Theorem}{Theorem}
\newtheorem{Proposition}{Proposition}
\newtheorem{Corollary}{Corollary}
\newtheorem{Lemma}{Lemma}

\newtheorem{prb}{Problem}

\newcommand{\R}{{\mathbb R}}
\newcommand{\Z}{{\mathbb Z}}
\newcommand{\X}{{\mathbb X}}
\newcommand{\Y}{{\mathcal Y}}

\begin{document}
	
\title{All finite sets are Ramsey in the maximum norm}
\author{Andrey Kupavskii\thanks{MIPT, Moscow, Russia; IAS, Princeton, USA; CNRS, Grenoble, France. Email: {\tt kupavskii@ya.ru}}, Arsenii Sagdeev\thanks{MIPT, Moscow, Russia. Email: {\tt sagdeev.aa@phystech.edu}}}
\date{}
	
\maketitle
\begin{abstract}
	For two metric spaces $\mathbb X$ and $\mathcal Y$, the chromatic number $\chi(\mathbb X;\mathcal Y)$ of $\mathbb X$ with forbidden $\mathcal Y$ is the smallest $k$ such that there is a coloring of the points of $\mathbb X$ with $k$ colors and no monochromatic copy of $\mathcal Y$. In this paper, we show that for each finite metric space $\mathcal{M}$ that contains at least two points the value $\chi\left(\mathbb R^n_\infty; \mathcal M \right)$ grows exponentially with $n$. We also provide explicit lower and upper bounds for some special $\mathcal M$.
	\end{abstract}
\textbf{MSC classification codes:} 05D10, 52C10	

\section{Introduction}

Ramsey theory is a central part of modern combinatorics with many connections to other areas, such as logic, number theory, and computer science. Its  topic is to find homogeneous substructures in sufficiently large or dense structures. The early examples include Schur's lemma, van der Waerden's theorem, and Ramsey's theorem. We refer the reader to the classical book \cite{GraRothSpen1990} for a survey of early developments of Ramsey theory. 

Most early results were concerned with finding homogeneous substructures in combinatorial structures (such as monochromatic complete subgraphs in any two-coloring of edges of a large complete graph). The problems we are interested in in this  paper are of geometric nature.

In 1950 Nelson posed the following question: what is the minimal number of colors needed to color all points of the Euclidean plane $\mathbb{R}^2$ such that no two points at unit distance apart receive the same color? This quantity is called {\it the chromatic number of the plane} and is denoted by $\chi( \mathbb{R}^2) $. This problem received great interest from the mathematics community, and even the origins of the questions spurred heated debates, partly because Nelson asked his question in private communication. We refer to the book of Soifer \cite{Soi2008_MathColBook} and surveys of Raigorodskii \cite{Rai2001, Rai2013} for a comprehensive account of the problem.

For almost 70 years, the problem resisted the attacks from numerous researchers, and the state of the art was at the easy-to-get lower and upper bounds $4\leq \chi(\R^2)\leq 7$ due to Nelson and Isbell respectively (or due to Mosers'  \cite{MM} and Hadwiger \cite{Had3}, see the aforementioned Soifer's historical research \cite{Soi2008_MathColBook}). Recently, the combinatorics community was shaken when an amateur mathematician de Grey \cite{deGrey} improved the lower bound to $\chi(\R^2)\geq 5$. A few weeks later, Exoo and Ismailescu \cite{ExIs2019} gave another proof of this bound. Since then, the initial de Grey's construction has been simplified several times within the framework of the Polymath16 Project (the current record is due to Parts~\cite{Par}).
	
One natural generalization of Nelson's problem concerns chromatic numbers of Euclidean spaces of other dimensions. One can find the current best lower and upper bounds on $\chi(\mathbb{R}^n)$ for several small values of $n$ in \cite{BogRai2019, CherKulRai2018}. As for $n$ large, Frankl and Wilson \cite{FranklWil1981} showed that the function $\chi( \mathbb{R}^n) $ grows exponentially with $n$. The best known asymptotic lower and upper bounds are  $\left(1.239...+o(1)\right)^n  \leq \chi(\mathbb{R}^n) \leq \left(3+o(1)\right)^n$ as $n \rightarrow \infty$ (see Larman and Rogers \cite{LarRog1972} or Prosanov \cite{Pros2020_LR} for the upper bound and  Raigorodskii \cite{Rai2000} for the lower bound).
	
A systematic study of such questions on the interface of geometry and Ramsey theory, named {\it Euclidean Ramsey theory}, began with three papers \cite{EGMRSS1,EGMRSS2,EGMRSS3} of Erd\H os, Graham, Montgomery, Rothschild, Spencer, and Straus.  Given a subset $S \subset \mathbb{R}^d$ (with induced metric), the value $\chi( \mathbb{R}^n; S)$ is defined to be the minimum number of colors needed to color all points of Euclidean space $\mathbb{R}^n$ with no monochromatic  isometric copy $S' \subset \mathbb{R}^n$ of $S$. Note this this definition only makes sense if  $|S|\ge 2$. In this case the value $\chi( \mathbb{R}^n; S)$ is well-defined, i.e., the corresponding minimum always exists, since we trivially have $\chi( \mathbb{R}^n; S) \le \chi(\mathbb{R}^n)$. 

A set $S \subset \mathbb{R}^d$ is called {\it $\ell_2-$Ramsey} if $\chi( \mathbb{R}^n; S) $ tends to infinity as $n \rightarrow \infty$. 
Similarly, a set $S$ is called {\it exponentially $\ell_2-$Ramsey} if there is a constant $\chi_S > 1$ such that $\chi( \mathbb{R}^n; S)  > \left( \chi_S+o(1)\right)^n$ as $n \rightarrow \infty$.
Relatively few sets are known to be exponentially $\ell_2-$Ramsey. Frankl and R\"odl \cite{FranklRodl1990} proved that the vertex sets of simplices and bricks (or hyperrectangles) are exponentially $\ell_2-$Ramsey. One can find several explicit exponential lower and upper bounds for these sets in  \cite{Nas2019_EqTri, Pros2018_ExpRams, Sag2018_CartProd, Sag2019_LotsEdges, SagRai2019_EuroComb}. More sets are known to have a weaker property of being $\ell_2-$Ramsey. K\v{r}\'\i\v{z} \cite{Kriz1991} proved that each `fairly symmetric' set is $\ell_2-$Ramsey. Later, this was used  by himself  \cite{Kriz1992} and Cantwell \cite{Cant2007} to show that the set of vertices of each regular polytope is $\ell_2-$Ramsey. Note that it is unknown if there is an $\ell_2-$Ramsey set that is not exponentially $\ell_2-$Ramsey.

At the same time, we know a strong necessary condition for a set to be $\ell_2-$Ramsey. Erd\H os et al. showed that each $\ell_2-$Ramsey set must be finite \cite{EGMRSS2} and {\it spherical} \cite{EGMRSS1}, i.e. be isometric to a subset of a sphere of some dimension. There is a popular conjecture stating that this is also sufficient. There is also a `rival' conjecture proposed in \cite{LeadRussWal2012} that states that only so-called `subtransitive' sets are $\ell_2-$Ramsey.

Another direction for generalizations that was explored is to work with other metrics than the Euclidean one. The natural candidates are the
$\ell_p-$metrics, defined for $\mathbf{x}, \mathbf{y} \in \mathbb{R}^n$ by $\|\mathbf{x}-\mathbf{y}\|_{p} = (\left|x_1-y_1\right|^p+\dots+\left|x_n-y_n\right|^p)^{1/p}$,
or  the {\it Chebyshev metric} (also known as the {\it maximum metric}) $\ell_\infty$ defined by
$\|\mathbf{x}-\mathbf{y}\|_{\infty} = \max_{1 \leq i \leq n} \left\{ \left|x_i-y_i\right|\right\}.$
We denote the corresponding spaces by $\R^n_{p}$ and the corresponding chromatic numbers by $\chi(\mathbb{R}^n_p)$. For all $1<p<\infty$, $p \neq 2$, the best known lower bound due to Frankl and Wilson \cite{FranklWil1981} is $ \chi(\mathbb{R}^n_p) \ge (1.207...+o(1))^n$ as $n \rightarrow \infty$, while in case $p=1$, i.e., in case of the {\it taxicab} or {\it Manhattan} metric, Raigorodskii \cite{Rai2004} proved a better lower bound $\chi(\mathbb{R}^n_1) \ge (1.366...+o(1))^n$. As for the upper bound, the best result for all real $p \neq 2$ due to the first author \cite{Kup2011} is $\chi(\mathbb{R}^n_p) \le (4+o(1))^n$ as $n \rightarrow \infty$ (also valid for any norm). The case $p = \infty$ stands out here because of the folklore equality $\chi( \mathbb{R}^n_\infty) = 2^n,$ valid  for each $n \in \mathbb{N}$. For completeness, we will give its simple proof in the next section.

The questions mentioned above can be described using the following general setup. Let $\X = \left( X, \rho_X\right), \Y = ( Y, \rho_Y)$ be two metric spaces. A subset $Y' \subset X$ is called a {\it copy} of $\Y$ if there is an {\it isometry} $f: Y \rightarrow Y'$, i.e., a bijection such that $\rho_Y(y_1,y_2) = \rho_X\big(f(y_1), f(y_2)\big)$ for all $y_1, y_2 \in Y$. The {\it chromatic number $\chi(\X;\Y)$ of the space $\X$ with a forbidden subspace $\mathcal{Y}$} is the minimal $k$ such that there is a coloring of elements of $X$ with $k$ colors and no monochromatic copy of $\Y$. 

In this paper, we focus on the case of $\X = \mathbb{R}^n_\infty$. Recall that for each  finite metric space $\mathcal{M}$ of size at least $2$ we trivially have $\chi(\mathbb{R}^n_\infty; \mathcal{M}) \le \chi(\mathbb{R}^n_\infty) = 2^n$ for all $n \in \mathbb{N}$. As for the lower bound, observe that it is not even obvious that $\chi(\mathbb{R}^n_\infty; \mathcal{M}) > 1$ for some $n$. This inequality follows form a result due to  Fr\'echet (see, e.g., \cite{Mat}) that states that every finite metric space can be embedded into $\mathbb{R}^n_\infty$ for some  $n \in \mathbb{N}$. We give the precise statement along with the short proof in Section~\ref{sec3} for completeness (see Lemma~\ref{L Frechet}). The following theorem, which  is the main result of our paper, gives a much stronger lower bound.

\begin{Theorem} \label{T1}
	Any finite metric space $\mathcal{M}$ that contains at least two points is exponentially $\ell_\infty-$Ramsey, i.e., there is a constant $\chi_\mathcal{M} > 1$ such that $\chi(\mathbb{R}^n_\infty; \mathcal{M}) > \left( \chi_\mathcal{M}+o(1)\right)^n$ as $n \rightarrow \infty$. 
\end{Theorem}

The rest of the present paper is organized as follows.  We start Section~\ref{sec2} by reciting the proof of $\chi( \mathbb{R}^n_\infty)=2^n$ and then prove Theorem~\ref{T1} for  `$1-$dimensional' metric spaces, called batons. This step is crucial in the proof of Theorem~\ref{T1}, and the results in that section are also of independent interest. In Section~\ref{sec3} we deduce Theorem \ref{T1} from the results of Section~\ref{sec2}. In Section~\ref{sec4} we provide upper bounds on the values $\chi(\mathbb{R}^n_\infty; \mathcal{M})$. Finally, Section~\ref{sec5} contains the discussion of the results and some open problems.

In what follows, whenever not specified, the distances are taken in the Chebyshev metric. We also slightly abuse notation and identify each set $S \subset \R^d$ with the corresponding metric space $\left(S, \ell_\infty\right)$ which is a subspace of  $\R^d_\infty$.

\section{Batons}\label{sec2}
For a $k \in \mathbb{N}$, denote $\left[k\right]_0 = \left\{0,1,  2,\, \dots \,,k\right\}$ (note that this is a slightly non-standard notation). Given a sequence of positive real numbers $\alpha_1,\, \dots \, , \alpha_k$,  a {\it baton} $\mathcal{B}(\alpha_1,\, \dots \, , \alpha_k)$ is a metric space isometric to a set of points $\left\{0, \alpha_1,  \alpha_1+\alpha_2,\, \dots \,,\sum_{i=1}^{k}\alpha_i \right\} \subset \mathbb{R}$ with the metric induced from $\mathbb{R}$. If  $\alpha_1=\dots=\alpha_k=1$ then we denote this space  $\mathcal{B}_k$ for shorthand.

\subsection{$\chi( \mathbb{R}^n_\infty)=2^n$}
We start the proofs with the simple but instructive case of $\chi( \mathbb{R}^n_\infty) = \chi( \mathbb{R}^n_\infty,\mathcal B_1)$. As we have already mentioned, the equality $\chi( \mathbb{R}^n_\infty)=2^n$ is folklore, and we give its proof for completeness.

Given $n \in \mathbb{N}$, let us denote $m=2^n$ for convenience. Let $\mathbf{v}_1, \,\dots\, , \mathbf{v}_m$ be a set of vertices of a standard (discrete) unit cube $\{0,1\}^n$. 
Note that $\|\mathbf{v}_i-\mathbf{v}_j\|_\infty = 1$ for all $i\neq j$, and thus $\chi( \mathbb{R}^n_\infty) \geq 2^n$ because we need to use a distinct color for each $\mathbf v_i$.

To prove the matching  upper bound, we explicitly describe the coloring. Let $$\mathcal{C} = \bigsqcup_{\mathbf{w}\in \Z^n} \big( \left[0;1\right)^n+2\mathbf{w} \big) $$ be a disjoint union of unit cubes. It is clear that for each $\mathbf{x},\mathbf{y} \in \mathcal{C}$ we have $\|\mathbf{x}-\mathbf{y}\|_\infty \neq 1$. Indeed, one has $\|\mathbf{x}-\mathbf{y}\|_\infty < 1$ whenever $\mathbf{x}$ and $\mathbf{y}$ come from the same unit cube, and $\|\mathbf{x}-\mathbf{y}\|_\infty > 1$ whenever they are from different cubes. Given $i \leq m$, let us denote $\mathcal{C}_i = \mathcal{C} + \mathbf{v}_i$ (where $\mathbf{v}_i$ were defined in the previous paragraph).  Color each point of $\mathcal{C}_i$ with the $i$'th color. We have $\bigsqcup_{i=1}^m \mathcal{C}_i = \R^n$, and thus it is a well-defined proper coloring of $\R^n$. This shows that $\chi( \mathbb{R}^n_\infty) \leq 2^n$.

\subsection{$\mathcal B_k$ is exponentially $\ell_\infty-$Ramsey}

\begin{Theorem} \label{B1}
	Let $k, n$ be positive integers. Then each subset $X \subset \left[k\right]_0^n \subset \mathbb{R}^n_\infty$ of cardinality $\left|X\right|>k^n$ contains a copy of $\mathcal{B}_k$.
\end{Theorem}
\begin{proof}
The proof is by induction on $n$. For $n=1$, there is nothing to prove. Indeed, if $X \subset \left[k\right]_0$ and $\left|X\right|> k$ then $X = \left[k\right]_0$ is the required copy of $\mathcal{B}_k.$

Next, assume that $n>1$. We employ a certain shifting-type argument. For a vector $\mathbf{x}=\left(x_1,\, \dots\,, x_n\right)\in \left[k\right]_0^n$ define its {\it head} $h\left( \mathbf{x} \right)=x_n$ and  {\it tail} $t(\mathbf{x})=\left(x_1,\, \dots\,, x_{n-1}\right)$. Given $\mathbf{y} \in \left[k\right]_0^{n-1}$, let $H(\mathbf{y}) = \left\{i \in [k]_0: \left(\mathbf{y}, i\right) \in X\right\}$. It should be clear that if $H(\mathbf{y}) = \left[k\right]_0$ for some $\mathbf y\in [k]_0^{n-1}$ then $\left\{ \left(\mathbf{y},0\right), \left(\mathbf{y}, 1\right),\,\dots\,,\left(\mathbf{y},k\right)\right\} \subset X$ is the required copy of $\mathcal{B}_k.$ In what follows, we assume that $H(\mathbf{y}) \neq \left[k\right]_0$ for all $\mathbf{y} \in \left[k\right]_0^{n-1}$.

Let us define a function $f:X\rightarrow\left[k\right]_0^n$ that increases the last coordinate of a vector by $1$ `whenever possible' as follows. For a vector $\mathbf x\in X$
\begin{equation*}
	f\left(\mathbf{x}\right) =
	\begin{cases}
		\left(t(\mathbf{x}),h(\mathbf x)+1\right) & \mbox{if } \exists\, j \in \left[k\right]_0\setminus H\left(t(\mathbf{x})\right) \mbox{ such that } j > h(\mathbf{x});\\
		\mathbf{x} & \mbox{otherwise.}
	\end{cases}
\end{equation*}
We show that $f(\cdot)$ is an injection. Indeed, it is clear that $f(\mathbf{x}^1)\neq f(\mathbf{x}^2)$ for all $\mathbf{x}^1, \mathbf{x}^2$ such that $t(\mathbf{x}^1) \neq t(\mathbf{x}^2)$ or $|h(\mathbf{x}^1)-h(\mathbf{x}^2)| \ge 2$. Thus, let us consider $\mathbf{y} \in [k]_0^{n-1}$, $0 \le i < k$, such that both $(\mathbf{y},i)$ and $(\mathbf{y},i+1)$ belong to $X$. It is not hard to see that either $f((\mathbf{y},i)) = (\mathbf{y},i+1), f((\mathbf{y},i+1)) = (\mathbf{y},i+2),$ or $f((\mathbf{y},i)) = (\mathbf{y},i), f((\mathbf{y},i+1)) = (\mathbf{y},i+1)$. In both cases we have $f((\mathbf{y},i))\neq f((\mathbf{y},i+1))$. Thus, $f(\cdot)$ is really an injection. Let $f\left(X\right)$ be the image of $X$ under $f(\cdot)$. Hence, 
\begin{equation} \label{B1 1}
\left|f(X)\right| = \left|X\right| > k^n.
\end{equation}

Partition $f\left(X\right)=X_0\,\sqcup\dots\,\sqcup X_{k}$ based on the last coordinate:
\begin{equation*}
X_i = \left\{f\left( \mathbf{x}\right) : \mathbf{x} \in X \mbox{ and } h\left( f\left( \mathbf{x}\right) \right) =i \right\}.
\end{equation*}
Since $H(\mathbf y)\ne [k]_0$ for all $\mathbf y\in [k]_0^{n-1}$, it is easy to see that $X_0$ is empty.
We conclude that
\begin{equation} \label{B1 2}
	\left|f(X)\right| = \left|X_1\right|+\dots+ \left|X_{k}\right|.
\end{equation}

It easily follows from comparing \eqref{B1 1} and \eqref{B1 2} that there is an $i \in \left\{1, \, \dots\,, k\right\}$ such that $\left|X_i\right|> k^{n-1}.$ Since distinct elements of $X_i$ have distinct tails, by the induction hypothesis one can find a set $\{\mathbf{y}^0,\,\dots\,,\mathbf{y}^k\}\subset\{t(\mathbf x): \mathbf x\in X_i\}$ that forms a copy of $\mathcal{B}_k$. For each $j \in [k]_0$, let $\mathbf{x}^j = f^{-1}(\mathbf{y}^j, i)\in X$ be the preimage of $(\mathbf{y}^j, i) \in X_i$. Note that we have $h(\mathbf{x}^j)\in \{i-1,i\}$ for each $j$, and thus $|h(\mathbf{x}^j ) -  h(\mathbf{x}^{j'} )|\in \{0,1\}$ for all $j,j' \in [k]_0$.  This immediately implies that $\|\mathbf{x}^j-\mathbf{x}^{j'}\|_\infty = \|\mathbf{y}^j- \mathbf{y}^{j'}\|_\infty$ for all $j,j' \in [k]_0$. Hence, the subset $\{ \mathbf{x}^0,\,\dots\,, \mathbf{x}^k\} \subset X$ is isometric to  $\mathcal{B}_k$.
\end{proof}

Theorem~\ref{B1} implies that if the coloring of $\R^n_\infty$ contains no monochromatic copy of $\mathcal{B}_k$ then each of its colors can intersect  $[k]_0^n$ in at most $k^n$  points. Using the pigeon-hole principle, we get the following corollary.

\begin{Corollary} \label{B2}
	For each positive integers $k,n$ one has $\chi(\mathbb{R}^n_\infty; \mathcal{B}_k) \geq \left(\frac{k+1}{k}\right)^n$. 
\end{Corollary}

\subsection{$\mathcal{B}\left( 1,\alpha\right)$ is exponentially $\ell_\infty-$Ramsey}

In this and the following two subsections we generalize Theorem~\ref{B1} and Corollary \ref{B2} to the case of arbitrary batons. However, the proof in the general case is nontrivial, and we wanted to illustrate some of its ideas on a much simpler case of  $\mathcal{B}\left( 1,\alpha\right)$, which is a set $\{0,1,1+\alpha\}$ with the natural metric.

\begin{Theorem} \label{B3}
	Let $\alpha>1$ be a real number. Then there is a subset $A\subset \mathbb{R}$ of cardinality $\left\lceil\alpha\right\rceil+2$ such that the following holds. Given a positive integer $n$, each subset $B \subset A^n \subset \mathbb{R}^n_{\infty}$ of cardinality $\left|B\right|>\left(\left\lceil\alpha\right\rceil+1\right)^n$ contains a copy of $\mathcal{B}\left(1,\alpha\right)$.
\end{Theorem}

\begin{proof}
Denote $m=\left\lceil\alpha\right\rceil$. 
Consider $A = \left\{a_0,\,\dots\,, a_{m+1}\right\} \subset \mathbb{R},$ where $a_l$ are defined as follows:
\begin{equation*}
a_0 = 0, \ \ a_l = 1+\frac{l-1}{m-1}\left(\alpha-1\right) \mbox{ for } 1 \leq l \leq m,
\ \ a_{m+1} = \alpha+1.
\end{equation*}
In particular, $a_1 = 1,$ $a_{m} = \alpha$. Define a bijection $f: [m+1]_0^n\to A^n$ by $f((x_1, \,\ldots\,,x_n)) = (a_{x_1}, \,\dots\,, a_{x_n})$. 

Let $B\subset A^n$ be an arbitrary subset of cardinality $\left|B\right|>\left(m+1\right)^n$. From Theorem \ref{B1} it follows that there is a subset $\left\{\mathbf{x}^0,\,\dots\,,\mathbf{x}^{m+1}\right\} \subset f^{-1}(B) \subset \left[m+1\right]_0^n$ that is a copy of $\mathcal{B}_{m+1}$. Without loss of generality, we can assume that $\|\mathbf{x}^s- \mathbf{x}^t\|_\infty = \left|s-t\right|$ for all $s,t \in \left[m+1\right]_0.$ In particular, it is easy to see that $\left\{\mathbf{x}^0, \mathbf{x}^1, \mathbf{x}^{m+1}\right\}$ is a copy of $\mathcal{B}\left(1,m\right)$. For convenience, let us denote $\mathbf{x}^0, \mathbf{x}^1,$ and $\mathbf{x}^{m+1}$ by $\mathbf{x}, \mathbf{y}$, and $\mathbf{z},$ respectively. Then the following two statements hold:

\begin{equation} \label{B3 0}
\mbox{for all } i \in \{1,\,\dots\,,n\} \mbox{ one has }
\begin{cases}
	\left|x_i-y_i\right| \leq 1, \\
	\left|y_i-z_i\right| \leq m;
\end{cases}
\end{equation}
\begin{equation} \label{B3 1}
\mbox{there is } j \in \{1,\,\dots\,,n\} \mbox{ such that either }
\begin{cases}
	x_j=0,\\
	y_j=1,\\
	z_j=m+1,
\end{cases}
\mbox{ or }
\begin{cases}
x_j=m+1,\\
y_j=m,\\
z_j=0.
\end{cases}
\end{equation}
The first statement is straightforward from the definition of $\mathcal B(1,m)$. As for the second, the equality $\|\mathbf{x}- \mathbf{z}\|_\infty = m+1$ implies that there is $j \in \{1,\,\dots\,,n\}$ such that $\left|x_j-z_j\right| = m+1$, and thus either  $x_j=0$ and $z_j=m+1$, or $x_j=m+1$ and $z_j=0$. Then it follows form \eqref{B3 0} that $y_j = 1$ in the former case and $y_j=m$ in the latter.

We claim that $\{f(\mathbf{x}), f(\mathbf{y}), f(\mathbf{z})\} \subset B$ is a copy of $\mathcal{B}\left(1,\alpha\right)$.  To check that, we need to verify that the distances between $f(\mathbf x), f(\mathbf y), f(\mathbf z)$ are the same as the distances between points in $\mathcal B(1,\alpha)$. 

Clearly, $\|\mathbf a^1-\mathbf a^2\|_\infty\leq \alpha+1$ for all $\mathbf a^1,\mathbf a^2\in A^n$. At the same time, it follows form \eqref{B3 1} that $|f(\mathbf x)_j-f(\mathbf z)_j|= |a_{x_j}-a_{z_j}| = \alpha+1,$ and thus $\|f(\mathbf x)-f(\mathbf z)\|_\infty=\alpha+1.$ 

Similarly, $|f(\mathbf x)_j-f(\mathbf y)_j| = 1$ and $|f(\mathbf y)_j-f(\mathbf z)_j| = \alpha,$ implying $\|f(\mathbf x)-f(\mathbf y)\|_\infty\geq 1$ and $\|f(\mathbf y)-f(\mathbf z)\|_\infty\geq \alpha.$ We actually have equality in both of these inequalities. Let us show it for the former, and the latter is analogous. Indeed, if $\|f(\mathbf x)-f(\mathbf y)\|_\infty> 1$  then there is $j' \in \{1,\,\dots\,,n\}$ such that $|f(\mathbf x)_{j'}-f(\mathbf y)_{j'}| > 1$. It follows from the definition of $a_l$ that if $|a_l-a_r| > 1$, then $|l-r|\geq 2$. Hence, $|x_{j'}-y_{j'}|\geq 2$, which contradicts \eqref{B3 0}.
\end{proof}

Observe that now one can easily deduce from Theorem~\ref{B3} that for all $\alpha>1$ and $n \in \mathbb{N}$, we have $\chi(\R_\infty^n; \mathcal{B}(1,\alpha)) \ge \left(\frac{\lceil \alpha \rceil+2}{ \lceil \alpha \rceil+1}\right)^n$. In particular, this implies that the metric space $\mathcal{B}(1,\alpha)$ is exponentially $\ell_\infty-$Ramsey.

\subsection{$\mathcal{B}\left( \alpha_1,\, \dots \, , \alpha_k\right)$  is exponentially $\ell_\infty-$Ramsey}

In this subsection we deal with the general case of   $\mathcal{B}\left( \alpha_1,\, \dots \, , \alpha_k\right)$. 
We use the same idea of reduction to the integer case and applying the pigeonhole principle. In the notation of the previous subsection, the main difficulty here is to find an appropriate $A$ and bijection $f$.

\begin{Theorem} \label{B5}
	Let $k$ be a positive integer and $\alpha_1,\, \dots \, , \alpha_k$ be positive real numbers. Then there is an integer $m$ and a subset $A\subset \mathbb{R}$ of cardinality $m+1$ such that the following holds. Given a positive integer $n$, each subset $B \subset A^n \subset \mathbb{R}^n_\infty$ of cardinality $\left|B\right|>m^n$
	contains a copy of $\mathcal{B}\left( \alpha_1,\, \dots \, , \alpha_k\right)$.
\end{Theorem}

First, we suppose that there exists an `appropriate' finite subset $A =\{a_0, \, \ldots \, ,a_m\} \subset \mathbb{R}$ that satisfies several conditions and deduce  Theorem \ref{B5} using it. Then we prove the existence of this `appropriate' $A$.

\begin{Lemma} \label{L2}
	Let $k$ be a positive integer and $\alpha_1,\, \dots \, , \alpha_k$ be positive real numbers. Set
	\begin{equation*}
		\Gamma = \left\{\gamma : \gamma \leq \alpha_1+\dots+ \alpha_k \mbox{ and } \gamma = d_1\alpha_1+\dots+ d_k\alpha_k \mbox{ for some } d_1,\,\dots\,, d_k \in \mathbb{N}\cup\left\{0\right\} \right\}.
	\end{equation*}
	Then there are positive integers $p_1,\, \dots \, , p_k$ and a sequence of real numbers $a_0 < \dots < a_m$, where $m=p_1+\dots+p_k$, such that the following two statements hold. First, for all positive integers $l$ and $r$ such that $l+r \leq m$ one has
	\begin{equation} \label{L2 1}
	a_{l+r} \leq a_l + a_r.
	\end{equation}
	Second, for each $\gamma = d_1\alpha_1+\dots+ d_k\alpha_k \in \Gamma$ one has $d_1p_1+\dots+ d_kp_k \leq m$ and
	\begin{equation} \label{L2 2}
	a_{d_1p_1+\dots+ d_kp_k} = d_1\alpha_1+\dots+ d_k\alpha_k.
	\end{equation}
\end{Lemma}

\begin{proof}[Proof of Theorem~\ref{B5}]
This proof essentially repeats the  argument from the previous subsection. Take $p_1, \,\ldots\, , p_k, m,$ and $a_0<\ldots<a_m$ as in Lemma~\ref{L2}. Set $A = \{a_0, \,\ldots\,, a_m\}$. Define a bijection $a: [m]_0 \to A$ by $a(l) = a_l$ for all $l \in [m]_0$. This function is strictly increasing by construction. Given $n \geq 1$, let $f: [m]_0^n\to A^n$ be a bijection defined by $f((x_1,\,\ldots\,,x_n)) = (a_{x_1}, \,\dots\,, a_{x_n})$. Note that we have $f(\mathbf{x})_i = a(x_i) = a_{x_i}$ for all $\mathbf{x} = (x_1,\,\ldots\,,x_n) \in [m]_0^n, 1 \leq i \leq n$.

Let $B \subset A^n$ be an arbitrary subset of cardinality $\left|B\right|>m^n$. From Theorem \ref{B1} it follows that there is a subset $\{\mathbf{x}^0,\,\dots\,,\mathbf{x}^{m}\} \subset f^{-1}(B) \subset \left[m\right]_0^n$ that is a copy of $\mathcal{B}_{m}$. Without loss of generality, we can assume that $\|\mathbf{x}^s- \mathbf{x}^t\|_\infty = \left|s-t\right|$ for all $s,t \in \left[m\right]_0.$ In particular, it is easy to see that $\{\mathbf{x}^0, \mathbf{x}^{p_1}, \mathbf{x}^{p_1+p_2},\, \dots \,, \mathbf{x}^{ p_1+p_2+\dots+p_k}\}$ is a copy of $\mathcal{B}\left( p_1,\, \dots \, , p_k\right)$. For convenience,  let us denote  $\mathbf{x}^0, \mathbf{x}^{p_1}, \mathbf{x}^{p_1+p_2},\, \dots \,, \mathbf{x}^{ p_1+p_2+\dots+p_k}$ by $\mathbf{y}^0, \mathbf{y}^{1}, \, \dots \,, \mathbf{y}^{k}$, respectively. They satisfy the following two properties:
\begin{equation} \label{L1 3}
\mbox{for all } 0 \leq s<t\leq k \mbox{ and for all } i \in \{1,\,\dots\,,n\} \mbox{ one has } \left|y_{i}^s-y_{i}^t\right| \leq p_{s+1}+\dots+p_{t},
\end{equation}
\begin{equation} \label{L1 4}
\mbox{there is } j \in \{1,\,\dots\,,n\} \mbox{ s.t. either }
\begin{cases}
y_{j}^0=0,\\
y_{j}^1=p_1,\\
\dots\\
y_{j}^k=p_1+\dots+p_k,
\end{cases}
\mbox{ or }
\begin{cases}
y_{j}^0=p_1+\dots+p_k,\\
y_{j}^1=p_1+\dots+p_{k-1},\\
\dots\\
y_{j}^k=0.
\end{cases}
\end{equation}
The first property immediately follows from $\|\mathbf y^s-\mathbf y^t\|_{\infty} = p_{s+1}+\dots+p_{t}$. As for the second, given that $\|\mathbf y^0-\mathbf y^k\|_{\infty} = m$ and $\mathbf y^0, \mathbf y^k \in [m]_0^n,$ we must have a $j$ such that either $y^0_j=0$ and $y^k_j=m,$ or $y^0_j=m$ and $y^k_j=0.$ In the rest of this subsection we assume that the former holds (the other case is symmetric). Then for each $1\leq s\leq k-1,$ given the distances $\|\mathbf y^0-\mathbf y^s\|_{\infty} = p_1+\ldots+p_s,$ $\|\mathbf y^s-\mathbf y^k\|_{\infty} = p_{s+1}+\ldots+p_k = m-\|\mathbf y^0-\mathbf y^s\|_{\infty},$ we clearly must have $y^s_j = p_1+\ldots+p_s.$

Using these two properties, we conclude the proof of Theorem~\ref{B5} by showing that the set $\{f(\mathbf{y}^0), f(\mathbf{y}^{1}), \, \dots \,, f(\mathbf{y}^{k})\} \subset B$ is a copy of  $\mathcal{B}\left( \alpha_1,\, \dots \, , \alpha_k\right)$.

On the one hand, given $0 \leq s<t\leq k$, we use \eqref{L2 2} and \eqref{L1 4} to get that $f(\mathbf{y}^s)_j = \alpha_1+\ldots+\alpha_s$ and $f(\mathbf{y}^t)_j = \alpha_{1}+\ldots+\alpha_t$. This implies that
\begin{equation} \label{B5 1}
	|f(\mathbf{y}^s)_j-f(\mathbf{y}^t)_j| = \alpha_{s+1}+\ldots+ \alpha_t.
\end{equation}

On the other hand, fix any $j'\in \{1, \,\dots\,, n\}$ and w.l.o.g. assume that $y^s_{j'}\geq y^t_{j'}.$ Using the monotonicity of $a(\cdot)$ and \eqref{L2 1}, we get that
\begin{equation} \label{B5 2}
0 \leq f(\mathbf{y}^s)_{j'} - f(\mathbf{y}^t)_{j'} = a(y^s_{j'}) - a(y^t_{j'}) = a(y^s_{j'}- y^t_{j'} +y^t_{j'}) - a(y^t_{j'}) \leq a(y^s_{j'}-y^t_{j'}).
\end{equation}
It follows from \eqref{L1 3} that $y^s_{j'}- y^t_{j'} \leq p_{s+1}+\dots+p_{t}$. Using the monotonicity of $a(\cdot)$ and \eqref{L2 2}, we get that
\begin{equation} \label{B5 3}
 a(y^s_{j'}-y^t_{j'}) \leq a(p_{s+1}+\dots+p_{t}) = \alpha_{s+1}+\dots+\alpha_t.
\end{equation}
We substitute \eqref{B5 3} in \eqref{B5 2} and get that $|f(\mathbf{y}^s)_{j'} - f(\mathbf{y}^t)_{j'}| \leq \alpha_{s+1}+\dots+\alpha_t$. Together with  \eqref{B5 1} this gives $\|f(\mathbf{y}^s)-f(\mathbf{y}^t)\|_{\infty} = \alpha_{s+1}+\ldots +\alpha_t$ for all $0 \leq s<t\leq k$. This implies that $\{f(\mathbf{y}^0), f(\mathbf{y}^{1}), \, \dots \,, f(\mathbf{y}^{k})\} \subset B$ is indeed a copy of  $\mathcal{B}\left( \alpha_1,\, \dots \, , \alpha_k\right)$.\end{proof}

It only remains  to prove Lemma \ref{L2} in order to finish the proof of Theorem \ref{B5}. We do this in a separate subsection. 

\subsection{The proof of Lemma \ref{L2}}

Suppose that $\Gamma = \{\gamma_0, \,\ldots\,, \gamma_t\}$, where $\gamma_0 < \dots < \gamma_t$. In particular, $\gamma_0 = 0, \gamma_1 = \min_{1\leq i \leq k} \{\alpha_i\}$, and $\gamma_t = \alpha_1+\dots+\alpha_k$. Let $\gamma_{t+1}$ be the smallest linear combination of $\alpha_1,\, \dots \, , \alpha_k$ with nonnegative integer coefficients that is greater than $\gamma_t$, i.e.,
\begin{equation*}
	\gamma_{t+1} = \min\big\{\gamma: \gamma > \gamma_t \mbox{ and } \gamma = d_1\alpha_1+\dots+d_k\alpha_k \mbox{ for some } d_1, \,\ldots\, , d_k \in \mathbb{N}\cup\{0\}\big\}.
\end{equation*}

Put $\delta = \min_{1\leq i \leq t+1}\{\gamma_{i}-\gamma_{i-1}\}$ and $\theta = \gamma_{t}/\gamma_1$. Let $q_0$ be a large enough integer such that
\begin{equation*}
	\frac{1}{q_0} < \delta \ \ \text{and} \ \ \frac{\theta}{q_0^{1+1/k}} < \frac{1}{2q_0}.
\end{equation*}

We will apply the following result of Dirichlet on Diophantine approximations (see, e.g., \cite{Schmidt1996}, Section~2, Theorem~1A).

\begin{Theorem} \label{Dir}
	Given $\alpha_1',\, \dots \, , \alpha_k'\in\mathbb R$ and $q_0'\in \mathbb N$, there is an integer $q > q_0'$ such that the following holds. There are  $p_1,\, \dots \, , p_k \in \mathbb{Z}$ such that for each $1 \leq i \leq k$ one has
	\begin{equation*}
	\left|\alpha_i'-\frac{p_i}{q}\right| < \frac{1}{q^{1+1/k}}.
	\end{equation*}
\end{Theorem}

Apply this theorem with $\alpha_i$ playing the role of $\alpha'_i$ and $q_0$  playing the role of $q_0'$. Let $q, p_1, \,\dots\,, p_k$ be as in the conclusion of this theorem. Given $\gamma \in \R$, denote by $c(\gamma) = \left\lfloor q\gamma \right\rceil$ the numerator of the best rational approximation of $\gamma$ with denominator equal to $q$. Note that $ |\gamma - \frac{c(\gamma)}{q}| \leq \frac{1}{2q}$ for all $\gamma$. Moreover, if $|\gamma - \frac{c}{q}| < \frac{1}{2q}$ for some $\gamma \in \R, c \in \mathbb{N}$ then $c = c(\gamma)$. The following two propositions show that the function $c(\cdot)$ is strictly increasing and `linear' on $\Gamma$.

\begin{Proposition} \label{P1}
	One has $c(\gamma_{i}) > c(\gamma_{i-1})$ for all $1 \leq i \leq t$.
\end{Proposition}
\begin{proof}
	Given $1 \leq i \leq t$, it is clear that $c(\gamma_{i}) \geq c(\gamma_{i-1})$. Assume that $c(\gamma_{i}) = c(\gamma_{i-1})$. Now it follows by the triangle inequality that
	\begin{equation*}
	\delta \leq \left|\gamma_{i}-\gamma_{i-1}\right| \leq \left|\gamma_{i}-\frac{c(\gamma_{i})}{q}\right| + \left|\gamma_{i-1}- \frac{c(\gamma_{i-1})}{q}\right| \leq \frac{1}{2q} + \frac{1}{2q} = \frac{1}{q} < \frac{1}{q_0},
	\end{equation*} 
	which contradicts the definition of $q_0$.
\end{proof}

\begin{Proposition} \label{P2}
	Let $\gamma = d_1\alpha_1+\dots+ d_k\alpha_k \in \Gamma$. Then $c(\gamma) = d_1p_1+\dots+ d_kp_k$. In particular, for all $\gamma_{i}, \gamma_{i'} \in \Gamma$ such that $\gamma_{i}+\gamma_{i'} \in \Gamma$ one has $c(\gamma_{i}+\gamma_{i'}) = c(\gamma_{i})+c(\gamma_{i'})$.
\end{Proposition}
\begin{proof}
	Set $ \varepsilon_i = \left|\alpha_i- \frac{p_i}{q}\right|$ and recall that $\varepsilon_i\leq q^{-(1+1/k)}$.
	Recall that
	\begin{equation*}
	   \theta= \frac{\gamma_t}{\gamma_1} = \frac{\alpha_1+ \dots+\alpha_k}{\min_{1\leq i \leq k} \left\{\alpha_i\right\}}.
	\end{equation*}
	Since $\gamma \leq \alpha_1+\dots+\alpha_k$, it is easy to see that $d_1+\dots+ d_k \leq \theta$. Now it is clear that
	\begin{equation*}
	\left|\gamma - \frac{d_1p_1+\dots+ d_kp_k}{q}\right| \leq d_1\varepsilon_1+\dots+ d_k\varepsilon_k \leq \frac{d_1+\dots+ d_k}{q^{1+1/k}}
	\leq \frac{\theta}{q^{1+1/k}} < \frac{1}{2q},
	\end{equation*}
	since $q > q_0$. This implies that $c(\gamma) = d_1p_1+\ldots+d_kp_k.$ The second part of the proposition is immediate from the first part by the `linearity' of $c(\gamma)$ on $\gamma\in \Gamma$.
\end{proof}

Proposition~\ref{P2} implies that $c(\gamma_0) = 0$ and $c(\gamma_{t}) = p_1+\dots+p_k=m$. We define the desired sequence $a_0, \,\dots\,, a_m$ as follows. Set $a_0 = 0$. Given $0 < l \leq m$, it follows from Proposition \ref{P1} that there is a unique $1\leq i \leq t$ such that $ c(\gamma_{i-1}) < l \leq c(\gamma_{i})$. We set
\begin{equation*}
a_l = \gamma_{i} - \frac{c(\gamma_{i})-l}{2m}\delta.
\end{equation*}

To finish the proof of Lemma \ref{L2} we  need to verify that this sequence is strictly increasing and satisfies \eqref{L2 1} and \eqref{L2 2}.

First, we show that $a_{l} > a_{l-1}$ for all $1 \leq l \leq m$. If there is $1 \leq i \leq t$ such that $c(\gamma_{i-1})< l-1 < l \leq c(\gamma_{i})$ then $a_{l} - a_{l-1} = \frac{1}{2m}\delta > 0$. Moreover, for all $1 \leq i \leq t$ one has
\begin{equation*}
	a_{c(\gamma_{i-1})+1}-a_{c(\gamma_{i-1})} = \left( \gamma_{i}- \frac{c(\gamma_{i}) -c(\gamma_{i-1})-1}{2m}\delta\right)-\gamma_{i-1} > (\gamma_{i} - \gamma_{i-1}) - \frac{\delta}{2} \geq \frac{\delta}{2} > 0.
\end{equation*}

Second, for all $0 \leq i \leq t$ we have $a_{c(\gamma_{i})} = \gamma_{i}$ by construction. Therefore, Proposition~\ref{P2} implies that the sequence $a_0, \,\dots\,, a_m$ satisfies \eqref{L2 2}.

Finally, given $0 \leq l,r \leq m$ such that $l+r \leq m$, we need to check that $a_{l+r} \leq a_l + a_r$. Observe that there is nothing to check if either $l=0$ or $r=0$. Thus w.l.o.g. we assume that both $l$ and $r$ are greater than $0$. Under this assumption there are unique $1 \leq i, i',j \leq t$ such that $c(\gamma_{i-1}) < l \leq c(\gamma_{i})$, $c(\gamma_{i'-1}) < r \leq c(\gamma_{i'})$, and $c(\gamma_{j-1}) < l+r \leq c(\gamma_{j})$. Set $\gamma = \gamma_{i}+\gamma_{i'}$.

Assume that $\gamma < \gamma_{j}$. Then $\gamma \in \Gamma$ and we can apply Proposition \ref{P2} to get that $c(\gamma_{j-1}) <l+r \leq c(\gamma_{i})+c(\gamma_{i'}) = c(\gamma)$. Thus, by Proposition~\ref{P1} we get that  $\gamma > \gamma_{j-1}$ and so $\gamma\geq \gamma_j$, a contradiction. Thus,  $\gamma \geq \gamma_{j}$.

Suppose that  $\gamma = \gamma_{j}$. Then Proposition \ref{P2} implies that $c(\gamma) = c(\gamma_{i})+ c(\gamma_{i'})$, and we have
\begin{align*}
a_l+a_r-a_{l+r} =& \left( \gamma_{i} - \frac{c(\gamma_{i})-l}{2m}\delta \right) + \left( \gamma_{i'} - \frac{c(\gamma_{i'})-r}{2m}\delta \right) - \left( \gamma_{j} - \frac{c(\gamma_{j})-l-r}{2m}\delta \right)  \\
=& (\gamma_{i}+\gamma_{i'}-\gamma)+\frac{c(\gamma)-c(\gamma_{i})- c(\gamma_{i'})}{2m}\delta = 0.
\end{align*}

Suppose that $\gamma > \gamma_{j}$ and, consequently,  $\gamma \geq \gamma_{j+1}$. The sequence of $a_i$'s is increasing, and so $a_{l+r} \leq a_{c(\gamma_{j})} = \gamma_{j}$. We conclude that 
\begin{align*}
a_l+a_r-a_{l+r} \geq& \left( \gamma_{i} - \frac{c(\gamma_{i})-l}{2m}\delta \right) + \left( \gamma_{i'} - \frac{c(\gamma_{i'})-r}{2m}\delta \right) - \gamma_{j}  \\
>& (\gamma_{i}+\gamma_{i'}-\gamma_{j}) -\frac{\delta}{2} -\frac{\delta}{2} \geq (\gamma_{j+1}-\gamma_{j}) -\delta \geq 0.
\end{align*}
This concludes the proof of Lemma~\ref{L2}.

\section{The proof of Theorem \ref{T1}}\label{sec3}

We begin with some notation that was introduced (in a slightly different form) by Frankl and R\"odl  \cite{FranklRodl1990}. Given a real $p \ge 1$ or $p=\infty$, a metric space $\mathcal{M}$ is called {\it $\ell_p-$super-Ramsey (with parameters $F_\mathcal{M}$ and $\chi_\mathcal{M}$)} if there exist constants $F_\mathcal{M} \geq \chi_\mathcal{M} > 1$ and a sequence of sets $V_\mathcal{M}(n) \subset \mathbb{R}^n_p$ such that $\left| V_\mathcal{M}(n)\right| \leq \left( F_\mathcal{M}+o(1) \right)^n$ and each subset of $V_{\mathcal M}(n)$ of size greater than $\left|V_{\mathcal M} (n)\right| \left( \chi_\mathcal{M}+o(1) \right)^{-n}$, $n \rightarrow \infty$, contains a copy of $\mathcal M$. An easy application of the pigeonhole principle as in Corollary~\ref{B2} shows that for each  $\ell_p-$super-Ramsey metric space $\mathcal{M}$ with parameters $F_\mathcal{M}$ and $\chi_\mathcal{M}$ that contains at least two points, one has $\chi(\mathbb{R}^n_p; \mathcal{M}) \geq \left( \chi_\mathcal{M}+o(1) \right)^n$ as $n \rightarrow \infty$. So, the $\ell_p-$super-Ramsey property implies the exponentially $\ell_p-$Ramsey one. Therefore, in order to prove Theorem~\ref{T1} it is sufficient to prove the following statement.

\begin{Theorem} \label{T2}
	Any finite metric space is $\ell_\infty-$super-Ramsey.
\end{Theorem} 

Given a real $p\ge 1$ and given two metric spaces $\mathcal{X} = \left( X, \rho_X\right)$ and $\mathcal{Y} = \left( Y, \rho_Y\right)$, their {\it Cartesian product} $\mathcal{X} \times \mathcal{Y}$ is the metric space $\left( X\times Y, \rho\right)$, where
\begin{equation*}
	\rho((x_1, y_1), (x_2, y_2)) = (\rho_X(x_1, x_2)^p + \rho_Y(y_1, y_2 )^p)^{1/p} 
\end{equation*}
for all $x_1,x_2 \in X$ and $y_1,y_2 \in Y$. In a special case $p=\infty$, to which we devote the present paper, we set
\begin{equation*}
	\rho((x_1, y_1 ), (x_2, y_2)) = \max\left\{ \rho_X(x_1, x_2), \rho_Y( y_1, y_2)  \right\}.
\end{equation*}
Frankl and R\"odl \cite{FranklRodl1990} showed that a Cartesian product of any two $\ell_2-$super-Ramsey finite metric spaces  is also $\ell_2-$super-Ramsey. Their proof actually works for the $\ell_p$ metric for all $p$. In case $p=\infty$ the proof is spelled out  in a paper by the second author  \cite{Sag2018_CartProd}, where he also gives an explicit dependence of the parameters.

\begin{Theorem}[\cite{FranklRodl1990}, Theorem 2.2; \cite{Sag2018_CartProd}, Proposition 1] \label{CartProd1}
	Let $\mathcal{X}$ and $\mathcal{Y}$ be $\ell_\infty-$super-Ramsey finite metric spaces. Then $\mathcal{X} \times \mathcal{Y}$ is also  $\ell_\infty-$super-Ramsey.
\end{Theorem}

Note that Theorem~\ref{B5} implies that each baton $\mathcal{B}$ is $\ell_\infty-$super-Ramsey (with the parameters $F_\mathcal{B} = m+1$ and $\chi_\mathcal{B} = \frac{m+1}{m}$, where $m$ is from Theorem~\ref{B5}). Thus, by Theorem~\ref{CartProd1}, we conclude that for all $d \in \mathbb{N}$, each {\it $d-$dimensional grid}, i.e., a Cartesian product of $d$ batons, is also $\ell_\infty-$super-Ramsey. It should be clear that the property of being $\ell_\infty-$super-Ramsey (with parameters $F$ and $\chi$) is hereditary with respect to taking subsets.  Since each finite subset $\{\mathbf x^0, \,\ldots\, ,\mathbf x^k\}\subset \R^d$ is a subset of the $d-$dimensional grid $\prod_{i=1}^d\{x^0_i, \,\ldots\, , x^k_i\}$, we get the following corollary. 
\begin{Corollary}
	Given $d \in \mathbb{N}$, any finite $S\subset \R^d_\infty$ is  $\ell_\infty-$super-Ramsey.
\end{Corollary}
The following simple proposition due to Fr\'echet (see \cite{Mat})  concludes the proof of Theorem~\ref{T2}.

\begin{Lemma} \label{L Frechet}
 Any metric space $\mathcal M = (M,\rho)$ with $|M| = d$ is isometric to a subset $S\subset \R^d_{\infty}$.   
\end{Lemma}
\begin{proof}
    Let $M = \{u_1,\ldots, u_d\}$. The desired set $S$ is given by the rows of the distance matrix $D= \big(\rho(u_i,u_j)\big)_{i,j=1}^d$. 
\end{proof}

\section{Upper bounds}\label{sec4}

Let $\mathcal{M} = \left(M, \rho_M\right)$ be a finite metric space that contains at least two points. Theorem \ref{T1} guarantees the existence of an exponential lower bound on $\chi(\mathbb{R}^n_\infty; \mathcal{M})$. Recall that we also have $\chi(\R^n_\infty;\mathcal M)\leq \chi(\R^n_{\infty})=2^n$. In this section we give a better upper bound. Essentially the same proof appeared in \cite{Pros2018_ExpRams}, \cite{Pros2020_LR}.

For a metric space $\mathcal M$, let $d( \mathcal{M})$ be the {\it diameter} of $\mathcal{M}$, i.e., the maximum over distances between pair of points of $\mathcal M$. Let $l( \mathcal{M}) > 0$ be the smallest number such that for each two points  $x,y \in M$ there are $z_0, z_1, \,\dots\,, z_t \in M$ such that $z_0=x, z_t = y$, and $\rho_M(z_i,z_{i+1}) \leq l$ for all $0 \le i \le t-1$.

In what follows, we denote the natural logarithm by $\log (\cdot)$.

\begin{Theorem} \label{U1}
	For each finite metric space $\mathcal{M}$  that contains at least two points, one has $$\chi(\mathbb{R}^n_\infty; \mathcal{M}) < (1+o(1))n\log n\left( 1+\frac{l(\mathcal M)}{d(\mathcal M)}\right)^n$$ as $n \rightarrow \infty$. 
\end{Theorem}

\begin{proof}
	
Fix arbitrary positive $d' < d(\mathcal M)$ and $l'>l(\mathcal M)$. Let $C = \left[0;d'\right]^n$ be an $n-$dimensional cube and put $\mathcal{C}=\bigsqcup_{\mathbf{v} \in \mathbb{Z}^n}\left( C+\left( d'+l'\right)\mathbf{v} \right) \subset \R^n$.

Assume that there is a subset $M' \subset \mathcal{C}$ that is a copy of $\mathcal M$ and consider any  $ \mathbf{x}, \mathbf{y} \in M'$. Let $\mathbf{z}_0, \mathbf{z}_1, \,\dots\,, \mathbf{z}_t \in M'$ be a sequence of points of $M'$ such that $\mathbf{z}_0=\mathbf{x}, \mathbf{z}_t = \mathbf{y}$, and $\|\mathbf{z}_i-\mathbf{z}_{i+1}\|_\infty \leq l(M)$ for all $0 \le i \le  t-1$. Since the distance between any two different translates of $C$ in $\mathcal{C}$ is at least $l' > l(\mathcal{M})$, we conclude that $\mathbf{z}_i$ and $\mathbf{z}_{i+1}$ belong to the same translate of $C$ for all $0 \le i \le t-1$. Hence, $\mathbf{x}$ and $\mathbf{y}$ belong to the same translate of $C$, and thus $M'$ lies entirely within the same translate of $C$. However, this is impossible, because the diameter of $C$ is equal to $d' < d(\mathcal{M})$. Thus, $\mathcal C$ does not contain a copy of $\mathcal M$.
	 
To conclude the proof we use the classical Erd\H{o}s--Rogers result \cite{ErRog1962} which states that $(1+o(1))n\log n\left(1+\frac{l'}{d'}\right)^n $ translates of $\mathcal C$ are sufficient to cover $\R^n$. Color all points of the $i$-th translate of $\mathcal C$ from the covering in one color. This is clearly a valid coloring. We also remark that $o(1)$ in the last formula does not depend on $d'$ and $l'$. Hence, letting $d' \rightarrow d$ and $l' \rightarrow l$ we obtain the claimed upper bound.\end{proof}

It is straightforward from the definition that $l(\mathcal{M}) \leq d(\mathcal{M})$ for all metric spaces $\mathcal{M}$. Note that the upper bound from Theorem~\ref{U1} is slightly worse than the trivial bound $\chi(\mathbb{R}^n_\infty; \mathcal{M}) \leq 2^n$ if $l(\mathcal{M}) = d(\mathcal{M})$, but is asymptotically better if $l(\mathcal{M}) < d(\mathcal{M})$.

Let $\mathcal{M}$ be a metric space such that $l(\mathcal{M}) < d(\mathcal{M})$ and $\frac{l(\mathcal{M})}{d(\mathcal{M})} \in \mathbb{Q}$. Applying a proper homothety, one can assume without loss of generality that $l(\mathcal{M}), d(\mathcal{M}) \in \mathbb{N}$. We can slightly improve the result of Theorem~\ref{U1} for such metric spaces using a simpler probabilistic argument.

\begin{Theorem} \label{U2}
	Let $\mathcal{M}= (M,\rho_M)$ be a finite metric space such that $|M| \ge 2$, $l(\mathcal{M})$ and $d(\mathcal{M})$ are integers satisfying $l(\mathcal{M}) < d(\mathcal{M})$. Then, we have $$\chi(\mathbb{R}^n_\infty; \mathcal{M}) < (1+o(1))n\log d(\mathcal{M})\left( 1+\frac{l(\mathcal M)}{d(\mathcal M)}\right)^n$$ as $n \rightarrow \infty$. 
\end{Theorem}

\begin{proof}
	For shorthand, denote $l = l(\mathcal{M}), d = d(\mathcal{M}),$ and $m=d+l \in \mathbb{N}$. Let $C = [0;d)^n$ and $\mathcal{C}=\bigsqcup_{\mathbf{v} \in \mathbb{Z}^n}\left( C+m\mathbf{v} \right) \subset \R^n$. As in the proof of Theorem~\ref{U1}, one can see that $\mathcal C$ does not contain a copy of $\mathcal M$, and we can use the same color for all of its points. It only remains to cover $\R^n$ using as few translates of $\mathcal{C}$ as possible.
	
	Set $s = \lfloor n\log d \left( \frac{m}{d}\right)^n \rfloor $ and let $\mathbf{v}_1, \,\dots\, ,\mathbf{v}_s$ be the elements of $[m-1]_0^n$ chosen uniformly and independently at random. Let $X \subset [m-1]_0^n$ be the subset consisting of all points $\mathbf{x} \in [m-1]_0^n$ that do not belong to any $\mathcal C + \mathbf{v}_i$, where $1 \leq i \leq s$.
	
	It is easy to see  that for $\mathbf{x} \in [m-1]_0^n$ and $1 \leq i \leq s$, we have $\Pr [\mathbf{x} \in \mathcal C + \mathbf{v}_i] = \left( \frac{d}{m} \right)^n$. Hence, from the mutual independence of the choice of $\mathbf v_i$'s it follows that 
	\begin{equation*}
		\Pr [\mathbf{x} \in X] = \left( 1-\Big( \frac{d}{m} \Big)^n\right)^s.
	\end{equation*}
	The linearity of the expectation gives
	\begin{equation*}
		\mathrm{E}[|X|] = \left( 1-\Big( \frac{d}{m} \Big)^n\right)^s m^n < \left( 1-\Big( \frac{d}{m} \Big)^n\right)^{n\log d \left( \frac{m}{d}\right)^n} m^n < e^{-n\log d} m^n = \left( \frac{m}{d} \right)^n.
	\end{equation*}
	
	Thus there is a way to fix the choice of $\mathbf{v}_1, \,\dots\, ,\mathbf{v}_s \in [m-1]_0^n$ such that $|X| \leq \left( \frac{m}{d} \right)^n$.	Consider the translates $\{\mathcal{C}+\mathbf{v}_i: 1 \leq i \leq s\} \cup \{\mathcal{C}+\mathbf{x}: \mathbf{x} \in X\}.$ Together, they cover all points of $[m-1]_0^n$ by construction. Moreover, they cover all points of $\R^n$ by periodicity. Finally, the number of translates we used is equal to $s+|X| = (1+o(1))n\log d\left( \frac{m}{d}\right)^n$ as  required. 
\end{proof}

\section{Concluding remarks and open problems}\label{sec5}

One of the key objects in this paper are the batons $\mathcal{B}_k$. It follows from Corollary~\ref{B2} and Theorem~\ref{U2} that   $\chi(\mathbb{R}^n_\infty; \mathcal{B}_k) = \left(\frac{k+1}{k} + o(1) \right)^n$ as $n \rightarrow \infty$. However, by calculating the exact values of $\chi(\mathbb{R}^n_\infty; \mathcal{B}_k)$ for some pairs of fixed small $n$ and $k$ we found that both the lower and upper bounds are not tight. 

For other batons the situation is much worse. For instance, given $\alpha > 1$, it follows from Theorems~\ref{B3} and \ref{U1} that $ \big( \frac{ \left\lceil \alpha \right\rceil +2 } {\left\lceil \alpha \right\rceil +1} \big)^n \leq \chi \left( \mathbb{R}^n_\infty; \mathcal B(1,\alpha) \right) \leq \big( 1+\frac{\alpha}{1+\alpha} +o(1) \big)^n$. Neither of these two bounds appear to be tight in general, and it is an interesting problem to determine the correct base of the exponent for this function.

For metric spaces other than batons Theorem \ref{T1} does not immediately give an explicit exponential lower bound. One may of course directly follow its proof to extract the lower bound on $\chi_\mathcal{M} > 1$ for each specific $\mathcal{M}$. The resulting bound would strongly depend on the dimension of the grid in which we embed $\mathcal{M}$, since each application of Theorem \ref{CartProd1} weakens the lower bound significantly. For example, a lower bound for $2-$dimensional grid $\mathcal{B}_2^2 = \mathcal{B}_2 \times \mathcal{B}_2$ one can extract from our proof is only $\chi\left( \mathbb{R}^n_\infty; \mathcal{B}_2^2\right)  \geq \left( 1.0667... +o(1)\right)^n$. We managed to generalize the argument from Section~\ref{sec2} to make it applicable not only for batons but also for different multidimensional grids. One of the results that we are able to obtain is that, for each fixed $k$ and $m$, one has $\chi(\mathbb{R}^n_\infty; \mathcal{B}_k^m) = \left(\frac{k+1}{k} + o(1) \right)^n$  as $n \rightarrow \infty$, i.e., the base of the exponent does not depend on $m$.
This result will be one of the subjects of a separate paper that is currently in preparation. 

Finally, let us explicitly state a question that arises in connection with the covering technique from Section~\ref{sec4}.

\begin{prb}
What is the smallest number $c(n)$ of translates of cubes $\{0,1\}^n$ that is needed to cover the torus $\Z_3^n$?  
\end{prb}
The best upper and lower bounds that we know of are $C_1 (3/2)^n\le c(n)\le C_2 n (3/2)^n$ for some constants $C_1,C_2$. Although the problem might appear simple, it seems to contain the difficulties that one is typically faced with when working with coverings as in Section~\ref{sec4}. 

\vspace{5mm}
\noindent {\bf Acknowledgments:}
We thank A.M.~Raigorodskii for his helpful suggestions and active interest in our work during the preparation of this paper. We also thank the referee for their helpful comments.

\vspace{5mm}
\noindent {\bf Conflicts of Interest:}
None.

\vspace{5mm}
\noindent {\bf Financial Support:}
The authors acknowledge the financial support from the Ministry of Education and Science of the Russian Federation in the framework of MegaGrant no. 075-15-2019-1926, the Russian Foundation for Basic Research grants no. 20-31-70039 and  no. 20-31-90009, and the Council for the Support of Leading Scientific Schools of the President of the
Russian Federation (grant no. N.Sh.-2540.2020.1). The research of the second author was supported in part by the  Simons Foundation and by the Moebius Contest Foundation for Young Scientists. The second author is a Young Russian Mathematics award winner and would like to thank its sponsors and jury.


\end{document}